\documentclass[a4paper,11pt,reqno]{article}

\usepackage[utf8]{inputenc}
\usepackage[T1]{fontenc}
\usepackage{lmodern}
\usepackage[english]{babel}
\usepackage{microtype}

\usepackage{amsmath,amssymb,amsfonts,amsthm}
\usepackage{mathtools,accents}
\usepackage{mathrsfs}
\usepackage{aliascnt}
\usepackage{braket}
\usepackage{bm}
\usepackage{esint}

\usepackage[a4paper,margin=3cm]{geometry}
\usepackage[citecolor=blue,colorlinks]{hyperref}

\usepackage{enumerate}
\usepackage{xcolor}
\usepackage{verbatim}

\newtheorem{theorem}{Theorem}
\newtheorem{lemma}{Lemma}
\newtheorem{definition}{Definition}
\newtheorem{prop}{Proposition}

\newtheorem{remark}{Remark}
\newtheorem{claim}{Claim}

\newcommand{\R}{\mathbb{R}}
\newcommand{\N}{\mathbb{N}}

\newcommand{\cV}{\mathcal{V}}
\newcommand{\cL}{\mathcal{L}}
\newcommand{\cE}{\mathcal{E}}
\newcommand{\cP}{\mathcal{P}}
\newcommand{\cI}{\mathcal{I}}
\newcommand{\cD}{\mathcal{D}}
\newcommand{\Ch}{{\sf Ch}}
\newcommand{\ov}{\overline}
\newcommand{\interior}{\mathring}

\DeclareMathOperator{\CD}{CD}

\DeclareMathOperator{\RCD}{RCD}

\DeclareMathOperator{\supp}{supp}
\DeclareMathOperator{\Lip}{Lip}

\makeatletter
\def\newaliasedtheorem#1[#2]#3{
  \newaliascnt{#1@alt}{#2}
  \newtheorem{#1}[#1@alt]{#3}
  \expandafter\newcommand\csname #1@altname\endcsname{#3}
}
\makeatother

\numberwithin{equation}{section}

\newcommand{\eps}{\varepsilon}

\let\phi\varphi

\newcommand{\dist}{\mathsf{d}}

\newcommand{\meas}{\mathfrak{m}}

\newcommand{\di}{\mathop{}\!\mathrm{d}}

\newcommand{\restr}{\raisebox{-.1618ex}{$\bigr\rvert$}}

\title{Weighted Sobolev Inequalities in $\CD$(0,N) spaces}
\author{David Tewodrose\footnote{CY Cergy Paris University, david.tewodrose@cyu.fr}}

\begin{document}

\maketitle

\begin{abstract}
In this note, we prove global weighted Sobolev inequalities on non-compact $\CD(0,N)$ spaces satisfying a suitable growth condition, extending to possibly non-smooth and non-Riemannian structures a previous result from \cite{Minerbe} stated for Riemannian manifolds with non-negative Ricci curvature. We use this result in the context of $\RCD(0,N)$ spaces to get a uniform bound of the corresponding weighted heat kernel via a weighted Nash inequality.
\end{abstract}



\tableofcontents

\section{Introduction}

Riemannian manifolds with non-negative Ricci curvature have strong analytic properties. Indeed, the doubling condition and the local $L^2$-Poincaré inequality are satisfied on such spaces, and they imply many important results, like the well-known Li-Yau Gaussian estimates for a class of Green functions including the heat kernel \cite{LiYau} or powerful local Sobolev inequalities and parabolic Harnack inequalities (see e.g.~\cite{Saloff-Coste}).

In the recent years, several classes of possibly non-smooth metric measure spaces containing the collection of Riemannian manifolds with non-negative Ricci curvature have been under investigation, both from a geometric and an analytic point of view. For instance, in the context of measure spaces endowed with a suitable Dirichlet form, Sturm proved existence and uniqueness of the fundamental solution of parabolic operators along with Gaussian estimates and parabolic Harnack inequalities \cite{Sturm95, Sturm96}, provided the doubling and Poincaré properties hold. Afterwards, general doubling spaces with Poincaré-type inequalities were studied at length by Hajlasz and Koskela \cite{HajlaszKoskela} who proved local Sobolev-type inequalities, a Trudinger inequality, a Rellich-Kondrachov theorem, and many related results.

Approximately a decade ago, Sturm \cite{Sturm06} and Lott and Villani \cite{LottVillani} independently proposed the curvature-dimension condition $\CD(0,N)$, for $N \in [1,+\infty)$, as an extension of non-negativity of the Ricci curvature and bound from above by $N$ of the dimension for possibly non-smooth metric measure spaces. Coupled with the infinitesimal Hilbertiannity introduced later on by Ambrosio, Gigli and Savaré \cite{AmbrosioGigliSavare14} to rule out non-Riemannian structures, the $\CD(0,N)$ condition leads to the stronger $\RCD(0,N)$ condition, where $\mathrm{R}$ stands for Riemannian.

The classes of $\CD(0,N)$ and $\RCD(0,N)$ spaces have been extensively studied over the past few years, and it is by now well-known that they both contain the measured Gromov-Hausdorff closure of the class of Riemannian manifolds with non-negative Ricci curvature and dimension lower than $N$, as well as Alexandrov spaces with non-negative generalized sectional curvature and locally finite and non-zero $n$-dimensional Hausdorff measure, $n$ being lower that $N$. Moreover, $\CD(0,N)$ spaces satisfy the doubling and Poincaré properties, and $\RCD(0,N)$ spaces are, in addition, endowed with a regular and strongly local Dirichlet form called Cheeger energy (see Section 2). Therefore, the works of Sturm \cite{Sturm95, Sturm96} imply existence and uniqueness of an heat kernel, which by the way satisfies Gaussian estimates, on $\RCD(0,N)$ spaces. 

One of the interest of the $\CD(0,N)$ and $\RCD(0,N)$ conditions, and of the more general $\CD(K,N)$ and $\RCD(K,N)$ conditions for arbitrary $K\in \R$, is the possibility of proving classical functional inequalities on spaces with rather loose structure thanks to optimal transport or gradient flow arguments. In this regard, Lott and Villani obtained in \cite[Th.~5.29]{LottVillani2} a global Sobolev-type inequality for $\CD(K,N)$ spaces with $K>0$ and $N\in(2,+\infty)$. Later on, in their striking work \cite[Th.~1.11]{CavallettiMondino2}, Cavaletti and Mondino proved a global Sobolev-type inequality with sharp constant for bounded essentially non-branching $\CD^*(K,N)$ spaces with $K\in \R$ and $N\in(1,+\infty)$; in case $K>0$ and $N>2$, they get the classical Sobolev inequality with sharp constant. This last inequality had been previously justified on $\RCD^*(K,N)$ spaces with $K>0$ and $N>2$ by Profeta \cite{Profeta}.


The aim of this note is to provide a new related analytic result, namely a global weighted Sobolev inequality, for certain non-compact $\CD(0,N)$ spaces with $N>2$. It is worth underlying that our result does not require the Riemannian synthetic condition $\RCD(0,N)$. Here and throughout the paper, if $(X,\dist,\meas)$ is a metric measure space, we write $B_r(x)$ for the ball of radius $r>0$ centered at $x \in X$, and $V(x,r)$ for $\meas(B_r(x))$. 

\begin{theorem}[Weighted Sobolev inequalities]\label{th:weightedSobolev}
Let $(X,\dist,\meas)$ be a $\CD(0,N)$ space with $N > 1$. Assume that there exists $1 < \eta \le N$ such that
\begin{equation}\label{eq:growthcondition}
0 < \Theta_{inf} := \liminf\limits_{r \to +\infty} \frac{V(o,r)}{r^\eta} \le \Theta_{sup}:= \limsup\limits_{r \to +\infty} \frac{V(o,r)}{r^\eta} < + \infty
\end{equation}
for some $o \in X$. Then for any $1 \le p < \eta$, there exists a constant $C>0$, depending only on $N$, $\eta$, $\Theta_{inf}$, $\Theta_{sup}$ and $p$, such that for any continuous function $u : X \to \R$ admitting an upper gradient $g \in L^p(X,\meas)$,
\[
\left( \int_{X} |u|^{p^*} d\mu  \right)^{\frac{1}{p^*}}\le C \left( \int_{X} g^p  d\meas \right)^{\frac{1}{p}}
\]
where $p^*= Np/(N-p)$ and $\mu$ is the measure absolutely continuous with respect to $\meas$ with density 
$w_o = V(o,\dist(o,\cdot))^{p/(N-p)} \dist(o,\cdot)^{-Np/(N-p)}$.
\end{theorem}


Theorem \ref{th:weightedSobolev} extends a result by Minerbe stated for $p=2$ on $n$-dimensional Riemannian manifolds with non-negative Ricci curvature \cite[Th.~0.1]{Minerbe}. The motivation there was that the classical $L^2$-Sobolev inequality does not hold on those manifolds which satisfy \eqref{eq:growthcondition} with $\eta<N=n$, see \cite[Prop.~2.21]{Minerbe}. This phenomenon also holds on some metric measure spaces including Finsler manifolds, see the forthcoming \cite{T} for related results.

Our proof is an adaptation of Minerbe's proof to the setting of $\CD(0,N)$ spaces and is based upon ideas of Grigor'yan and Saloff-Coste introduced in the smooth category \cite{Grigor'yanSaloff-Coste} which extend easily to the setting of metric measure spaces. More precisely, we apply an abstract process (Theorem \ref{th:patching}) which permits to patch local inequalities into a global one by means of an appropriate discrete Poincaré inequality. In the broader context of metric measure spaces with a global doubling condition, a local Poincaré inequality, and a reverse doubling condition weaker than \eqref{eq:growthcondition}, this method provides ``adimensional'' weighted Sobolev inequalities, as explained in the recent work \cite{Tewodrose2}.

After that, we follow a classical approach (see e.g.~\cite{BakryCoulhonLedouxSaloff-Coste}) which was neither considered in \cite{Minerbe} nor in the subsequent related work \cite{Hein} to deduce a weighted Nash inequality (Theorem \ref{th:weightedNash}) for $\CD(0,N)$ spaces satisfying the growth assumption \eqref{eq:growthcondition}, provided $\eta >2$. Let us mention that in the context of non-reversible Finsler manifolds, Ohta put forward an unweighted Nash inequality \cite{Ohta17} and that Bakry, Bolley, Gentil and Maheux introduced weighted Nash inequalities in the study of possibly non-ultracontractive Markov semigroups \cite{BakryBolleyGentilMaheux}, but these inequalities seem presently unrelated to our.

We conclude this note with a natural consequence in the setting of $\RCD(0,N)$ spaces satisfying a uniform local Ahlfors regularity property, namely a uniform bound for the weighted heat kernel associated with a suitable modification of the Cheeger energy. To the best knowledge of the author, this is the first appearance of this weighted heat kernel whose properties would require a deeper investigation.

The paper is organized as follows. In Section 2, we introduce the tools of non-smooth analysis that we shall use throughout the article.~We also define the $\CD(0,N)$ and $\RCD(0,N)$ conditions, and present the aforementioned patching process. Section 3 is devoted to the proof of Theorem \ref{th:weightedSobolev}. Section 4 deals with the weighted Nash inequality and the uniform bound on the weighted heat kernel we mentioned earlier. The final Section 5 provides a non-trivial non-smooth space to which our main theorem applies.

\section*{Acknowledgments}
I warmly thank T.~Coulhon who gave the initial impetus to this work. I am also greatly indebted towards L.~Ambrosio for many relevant remarks at different stages of the work. Finally, I would like to thank V.~Minerbe for useful comments, G.~Carron and N.~Gigli for helpful final conversations, and the anonymous referees for precious suggestions.

\section{Preliminaries}

Unless otherwise mentioned, in the whole article $(X,\dist,\meas)$ denotes a triple where $(X,\dist)$ is a proper, complete and separable metric space and $\meas$ is a Borel measure, positive and finite on balls with finite and non-zero radius, such that $\supp(\meas)=X$. We use the standard notations for function spaces: $C(X)$ for the space of $\dist$-continuous functions, $\Lip(X)$ for the space of $\dist$-Lipschitz functions and $L^p(X,\meas)$ (respectively $L^p_{loc}(X,\meas)$) for the space of $p$-integrable (respectively locally $p$-integrable) functions, for any $1 \le p \le +\infty$. If $U$ is an open subset of $X$, we denote by $C_c(U)$ the space of continuous functions on $X$ compactly supported in $U$. We also write $L^0(X,\meas)$ (respectively $L^0_+(X,\meas)$) for the space of $\meas$-measurable (respectively non-negative $\meas$-measurable) functions. If $A$ is a subset of $X$, we denote by $\overline{A}$ its closure. For any $x \in X$ and $r>0$, we write $S_r(x)$ for $\overline{B_r(x)}\backslash B_r(x)$. For any $\lambda >0$, if $B$ denotes a ball of radius $r>0$, we write $\lambda B$ for the ball with same center as $B$ and of radius $\lambda r$.  If $A$ is a bounded Borel subset of $X$, then for any locally integrable function $u:X \to \R$, we write $u_A$ or $\fint_A u \di \meas$ for the mean value $\frac{1}{\meas(A)} \int_A u \di \meas$, and $\langle u \rangle_A$ for the mean value $\frac{1}{\mu(A)}\int_A u \di \mu$, where $\mu$ is as in Theorem \ref{th:weightedSobolev}.

Several constants appear in this work. For better readability, if a constant $C$ depends only on parameters $a_1, a_2, \cdots$ we always write $C=C(a_1,a_2,\cdots)$ for its first occurrence, and then write more simply $C$ if there is no ambiguity.\\

\textbf{Non-smooth analysis.}

\noindent Let us recall that a continuous function $\gamma : [0,L] \to X$ is called a rectifiable curve if its length
\[
L(\gamma):=\sup \left\{ \sum_{i=1}^n \dist(\gamma(x_i),\gamma(x_{i-1})) \, \, : \, \, 0 = x_0 < \cdots < x_n = L, \, \,  n \in \N\backslash\{0\} \right\}
\]
is finite. If $\gamma : [0,L] \to X$ is rectifiable then so is its restriction $\gamma\restr_{[t,s]}$ to any subinterval $[t,s]$ of $[0,L]$; moreover, there exists a continuous function $\bar{\gamma} : [0,L(\gamma)] \to X$, called arc-length parametrization of $\gamma$, such that $L(\bar{\gamma}\restr_{[t,s]})=|t-s|$ for all $0 \le t \le s \le L(\gamma)$, and a non-decreasing continuous map $\varphi : [0,L] \to [0,L(\gamma)]$, such that $\gamma = \bar{\gamma} \circ \varphi$ (see e.g. \cite[Prop. 2.5.9]{BuragoBuragoIvanov}). When $\gamma=\bar{\gamma}$, we say that $\gamma$ is parametrized by arc-length.

In the context of metric analysis, a weak notion of norm of the gradient of a function is available and due to Heinonen and Koskela \cite{HeinonenKoskela}. 

\begin{definition}[Upper gradients]\label{def:uppergradient}
Let $u : X \to [-\infty,+\infty]$ be an extended real-valued function. A Borel function $g : X \to [0,+\infty]$ is called upper gradient of $u$ if for any rectifiable curve $\gamma:[0,L] \rightarrow X$ parametrized by arc-length,
\[
|u(\gamma(L))-u(\gamma(0))| \le \int_0^L g(\gamma(s)) \di s.
\]
\end{definition}

Building on this, one can introduce the so-called Cheeger energies and the associated Sobolev spaces $H^{1,p}(X,\dist,\meas)$, where $p\in[1,+\infty)$, in the following way:




\begin{definition}[Cheeger energies and Sobolev spaces]\label{def:Sobolev}
Let $1 \le p < +\infty$. The $p$-Cheeger energy of a function $u \in L^p (X,\meas)$ is set as
\[
\mathrm{Ch}_p(u):=\inf \liminf\limits_{i \to \infty} \|g_{i}\|_{L^{p}}^p.
\]
where the infimum is taken over all the sequences $(u_i)_i \subset L^p(X,\meas)$ and $(g_{i})_i \subset L^0_+(X,\meas)$ such that $g_{i}$ is an upper gradient of $u_{i}$ and $\| u_i - u \|_{L^p} \to 0$. The Sobolev space $H^{1,p}(X,\dist,\meas)$ is then defined as the closure of $\Lip(X) \cap L^p(X,\meas)$ with respect to the norm 
\[
\|u\|_{H^{1,p}} := \left( \|u\|_{L^{p}}^p + \mathrm{Ch}_p(u) \right)^{1/p}.
\]
\end{definition}

\begin{remark}
Following a classical convention, we call Cheeger energy the $2$-Cheeger energy and write $\Ch$ instead of $\Ch_2$.
\end{remark}

The above relaxation process can be performed with slopes of bounded Lipschitz functions instead of upper gradients, see Lemma \ref{lem:standardlemma}. Recall that the slope of a Lipschitz function $f$ is defined as
\[
|\nabla f|(x) := 
\begin{cases}
\limsup\limits_{y \to x} \frac{|f(x)-f(y)|}{d(x,y)} & \text{if $x \in X$ is not isolated},\\
\quad \quad \quad \quad 0 & \text{otherwise},
\end{cases}
\]
and that it satisfies the chain rule, namely $|\nabla (f g)| \le f |\nabla g| + g |\nabla f|$ for any $f, g \in \Lip(X)$.

Let us recall that $(X,\dist,\meas)$ is called doubling if there exists $C_D\ge 1$ such that
\begin{equation}\label{doubling}
V(x,2r) \le C_{D} V(x,r) \quad \quad \forall x \in X, \, \forall r>0,
\end{equation}
and that it satisfies a uniform weak local $L^p$-Poincaré inequality, where $p \in [1,+\infty)$, if there exists $\lambda>1$ and $C_P>0$ such that
\begin{align}\label{local Poincaré}
\int_{B} |u-u_{B}|^{p} \di \meas \le C_{P} r^{p} \int_{\lambda B} g^{p} \di \meas \qquad
\end{align}
holds for any ball $B$ of arbitrary radius $r>0$, any $u \in L^1_{loc}(X,\meas)$ and any upper gradient $g \in L^p(X,\meas)$ of $u$. If \eqref{local Poincaré} holds with $\lambda=1$, we say that a uniform \textit{strong} local $L^p$-Poincaré inequality holds.


The next notion serves to turn weak inequalities into strong inequalities, see e.g.~\cite[Sect.~9]{HajlaszKoskela}.

\begin{definition}[John domain]\label{def:Johndomain}
A bounded open set $\Omega\subset X$ is called a John domain if there exists $x_{o} \in \Omega$ and $C_J>0$ such that for every $x \in \Omega$, there exists a Lipschitz curve $\gamma:[0,L] \rightarrow \Omega$ parametrized by arc-length such that $\gamma(0)=x$, $\gamma(L)=x_{o}$ and $t^{-1} \dist(\gamma(t), X \backslash \Omega) \ge C_J$ for any $t \in [0,L]$.
\end{definition}


Finally let us introduce a technical property taken from \cite{HajlaszKoskela}. For any $v \in L^0(X,\meas)$ and $0<t_1<t_2<+\infty$, we denote by $v_{t_{1}}^{t_{2}}$ the truncated function $\min(\max(0,v-t_{1}),t_{2}-t_{1})+t_1$. We write $\chi_A$ for the characteristic function of a set $A\subset X$.

\begin{definition}[Truncation property]
We say that a pair of $\meas$-measurable functions $(u,g)$ such that for some $p\in[1,+\infty)$, $C_P>0$ and $\lambda>1$, the inequality \eqref{local Poincaré} holds for any ball $B$ of arbitrary radius $r>0$, has the truncation property if for any $0<t_1<t_2<+\infty$, $b \in \R$ and $\eps \in \{-1,1\}$, there exists $C>0$ such that \eqref{local Poincaré} holds for any ball $B$ of arbitrary radius $r>0$ with $u$, $g$ and $C_P$ replaced by $(\eps(u-b))_{t_1}^{t_2}$, $g \chi_{\{t_1<u<t_2\}}$ and $C$ respectively.
\end{definition}

The next proposition is a particular case of \cite[Th.~10.3]{HajlaszKoskela}.

\begin{prop}\label{truncation}
If $(X, \dist,\meas)$ satisfies a uniform weak local $L^1$-Poincaré inequality, any pair $(u,g)$ where $u \in C(X)$ and $g\in L^1_{loc}(X,\meas)$ is an upper gradient of $u$ has the truncation property.\\
\end{prop}


\textbf{The $\CD(0,N)$ and $\RCD(0,N)$ conditions.}

\noindent Let us give the definition of the curvature-dimension conditions $\CD(0,N)$ and $\RCD(0,N)$. For the general condition $\CD(K,N)$ with $K \in \R$, we refer to \cite[Chap. 29 \& 30]{Villani}.

Recall that a curve $\gamma : [0,1] \rightarrow X$ is called a geodesic if $\dist(\gamma(s),\gamma(t))=|t-s| \dist(\gamma(0),\gamma(1))$ for any $s,t \in [0,1]$. The space $(X,\dist)$ is called geodesic if for any couple of points $(x_{0},x_{1}) \in X^2$ there exists a geodesic $\gamma$ such that $\gamma(0)=x_{0}$ and $\gamma(1)=x_{1}$. We denote by $\cP(X)$ the set of probability measures on $X$ and by $\cP_{2}(X)$ the set of probability measures $\mu$ on $X$ with finite second moment, i.e. such that there exists $x_o \in X$ for which $\int_X \dist^2(x_o,x) \di \mu(x) < + \infty$. The Wasserstein distance between two measures $\mu_{0},\mu_{1} \in \mathcal{P}_{2}(X)$ is by definition
\[
W_{2}(\mu_{0},\mu_{1}) := \inf \left( \int_{X \times X} \dist(x_{0},x_{1})^2 \di \pi(x_{0},x_{1})  \right)^{1/2}
\]
where the infimum is taken among all the probability measures $\pi$ on $X\times X$ with first marginal equal to $\mu_{0}$ and second marginal equal to $\mu_{1}$. A standard result of optimal transport theory states that if the space $(X,\dist)$ is geodesic, then the metric space $(\mathcal{P}_{2},W_{2})$ is geodesic too. Let us introduce the Rényi entropies.

\begin{definition}[Rényi entropies]
Given $N \in (1,+\infty)$, the $N$-Rényi entropy relative to $\meas$, denoted by $S_{N}(\cdot | \meas)$, is defined as follows:
\[
S_{N}(\mu | \meas) := - \int_{X} \rho^{1-\frac{1}{N}} \di \meas \qquad \forall \mu \in \cP(X),
\]
 where $\mu = \rho \meas + \mu^{sing}$ is the Lebesgue decomposition of $\mu$ with respect to $\meas$.
\end{definition}


We are now in a position to introduce the $\CD(0,N)$ condition, which could be summarized as weak geodesical convexity of all the $N'$-Rényi entropies with $N' \ge N$.

\begin{definition}[$\CD(0,N)$ condition]\label{def:CD(0,N)}
Given $N \in (1,+\infty)$, a complete, separable, geodesic metric measure space $(X,\dist,\meas)$ satisfies the $CD(0,N)$ condition if for any $N' \ge N$, the $N'$-Rényi entropy is weakly geodesically convex, meaning that for any couple of measures $(\mu_{0},\mu_{1}) \in \mathcal{P}_{2}(X)^2$, there exists a $W_{2}$-geodesic $(\mu_{t})_{t \in [0,1]}$ between $\mu_{0}$ and $\mu_{1}$ such that for any $t \in [0,1]$,
\[
S_{N}(\mu_t | \meas) \le (1-t) S_{N}(\mu_0 | \meas) + t S_{N}(\mu_1 | \meas).
\]
Any space satisfying the $\CD(0,N)$ condition is called a $\CD(0,N)$ space.
\end{definition}

The Bishop-Gromov theorem holds on $\CD(0,N)$ spaces \cite[Th.~30.11]{Villani}, and as a direct consequence, the doubling condition \eqref{doubling} holds too, with $C_D=2^N$. Moreover, Rajala proved the following uniform weak local $L^1$-Poincaré inequality \cite[Th.~1.1]{Rajala}.

\begin{prop}\label{prop:Rajala}
Assume that $(X,\dist,\meas)$ is a $\CD(0,N)$ space. Then for any function $u \in C(X)$ and any upper gradient $g \in L^{1}_{loc}(X,\meas)$ of $u$, for any ball $B \subset X$ of arbitrary radius $r>0$,
\[
\int_{B}|u-u_{B}|\di \meas \le 4 r \int_{2B} g \di \meas.
\] 
\end{prop}

The $\CD(0,N)$ condition does not distinguish between Riemannian-like and non-Riemannian-like structures: for instance, $\R^n$ equipped with the distance induced by the $L^\infty$-norm and the Lebesgue measure satisfies the $\CD(0,N)$ condition (see the last theorem in \cite{Villani}), though it is not a Riemannian structure because the $L^\infty$-norm is not induced by any scalar product. To focus on Riemannian-like structures, Ambrosio, Gigli and Savaré added to the theory the notion of infinitesimal Hilbertianity, leading to the so-called $\RCD$ condition, $R$ standing for \textit{Riemannian} \cite{AmbrosioGigliSavare14}.

\begin{definition}[$\RCD(0,N)$ condition]\label{def:RCD(0,N)}
$(X,\dist,\meas)$ is called infinitesimally Hilbertian if $\Ch$ is a quadratic form. If in addition $(X,\dist,\meas)$ is a $\CD(0,N)$ space, it is said to satisfy the $\RCD(0,N)$ condition, or more simply it is called a $\RCD(0,N)$ space.
\end{definition}

Let us provide some standard facts taken from \cite{AmbrosioGigliSavare14, Gigli}.
First, note that $(X,\dist,\meas)$ is infinitesimally Hilbertian if and only if $H^{1,2}(X,\dist,\meas)$ is a Hilbert space, whence the terminology. Moreover, for infinitesimally Hilbertian spaces, a suitable diagonal argument  justifies for any $f \in H^{1,2}(X,\dist,\meas)$ the existence of a function $|\nabla f|_{*} \in L^2(X,\meas)$, called \textit{minimal relaxed slope} or \textit{minimal generalized upper gradient} of $f$, which gives integral representation of $\Ch$, meaning:
$$
\Ch(f)=\int_X |\nabla f|_{*}^2 \di \meas \qquad \forall f \in H^{1,2}(X,\dist,\meas).
$$
The minimal relaxed slope is a local object, meaning that $|\nabla f|_* = |\nabla g|_*$ $\meas$-a.e.~on $\{f = g\}$ for any $f, g \in H^{1,2}(X,\dist,\meas)$, and it satisfies the chain rule, namely $|\nabla (fg)|_* \le f |\nabla g|_* + g |\nabla f|_*$ $\meas$-a.e.~on $X$ for all $f,g \in H^{1,2}(X,\dist,\meas)$.
In addition, the function 
$$
\langle\nabla f_1,\nabla f_2\rangle:=\lim_{\epsilon\to 0}\frac{|\nabla (f_1+\epsilon f_2)|_*^2-|\nabla f_1|_*^2}{2\epsilon}
$$
provides a symmetric bilinear form on $H^{1,2}(X,\dist,\meas)\times H^{1,2}(X,\dist,\meas)$ with values in $L^1(X,\meas)$, and
$$
\Ch (f_1,f_2) := \int_X \langle \nabla f_1, \nabla f_2 \rangle \di \meas \qquad \forall f_1, f_2 \in H^{1,2}(X,\dist,\meas),
$$
defines a strongly local, regular and symmetric Dirichlet form. Finally, the infinitesimally Hilbertian condition allows to apply the general theory of gradient flows on Hilbert spaces, ensuring the existence of the $L^2$-gradient flow $(h_{t})_{t\ge 0}$ of the convex and lower semicontinuous functional $\Ch$, called \textit{heat flow} of $(X,\dist,\meas)$. This heat flow is a linear, continuous, self-adjoint and Markovian contraction semigroup in $L^2(X,\meas)$. The terminology `heat flow' comes from the characterization of $(h_{t})_{t \ge 0}$ as the only semigroup of operators such that $t\mapsto h_t f$ is locally absolutely continuous in $(0,+\infty)$ with values in $L^2(X,\meas)$ and
\[
\frac{\di}{\di t}h_t f=\Delta h_t f\quad\text{for $\cL^1$-a.e. $t\in (0,+\infty)$}
\]
holds for any $f\in L^2(X,\meas)$, the Laplace operator $\Delta$ being defined in this context by:
\[
f\in D(\Delta)\,\,\,\Longleftrightarrow\,\,\,
\exists h:=\Delta f\in L^2(X,\meas)\,\,\text{s.t. } \Ch(f,g) = - \int_X hg\, \di \meas \,\,\,\forall g\in H^{1,2}(X,\dist,\meas).\\
\]


\textbf{Patching process}

\noindent Let us present now the patching process \cite{Grigor'yanSaloff-Coste, Minerbe} that we shall apply to get Theorem \ref{th:weightedSobolev}. In the whole paragraph, $(X,\dist)$ is a metric space equipped with two Borel measures $\meas_1$ and $\meas_2$ both finite and nonzero on balls with finite and nonzero radius and such that $\supp(\meas_1) = \supp(\meas_2)=X$. For any bounded Borel set $A \subset X$ and any locally $\meas_2$-integrable function $u:X \to \R$, we denote by $\{u\}_A$ the mean value $\frac{1}{\meas_2(A)} \int_A u \di \meas_2$. For any given set $S$, we denote by $Card(S)$ its cardinality. 

\begin{definition}[Good covering]\label{def:goodcovering}
 Let $A \subset A^{\#} \subset X$ be two Borel sets. A countable family $(U_{i},U_{i}^{*},U_{i}^{\#})_{i \in I}$ of triples of Borel subsets of $X$ with finite $\meas_j$-measure for any $j\in\{1,2\}$ is called a good covering of $(A,A^{\#})$ with respect to $(\meas_1,\meas_2)$ if:
\begin{enumerate}
\item for every $i \in I$, $U_{i} \subset U_{i}^{*} \subset U_{i}^{\#}$;
\item there exists a Borel set $E \subset A$ such that $A \backslash E \subset \bigcup_{i} U_{i} \subset \bigcup_{i} U_{i}^{\#} \subset A^{\#}$ and $\meas_1(E) = \meas_2 (E) = 0$;

\item\label{3} 
there exists $Q_{1} > 0$ such that $Card(\{ i \in I : U_{i_{0}}^{\#} \cap U_{i}^{\#} \neq \emptyset \}) \le Q_{1}$ for any $i_{0} \in I$;
\item\label{4} 
for any $(i,j) \in I \times I$ such that $\overline{U_{i}} \cap \overline{U_{j}} \neq \emptyset$, there exists $k(i,j) \in I$ such that $U_{i} \cup U_{j} \subset U_{k(i,j)}^{*};$
\item\label{5} 
there exists $Q_{2} > 0$ such that for any $(i,j) \in I \times I$ satisfying $\overline{U_{i}} \cap \overline{U_{j}} \neq \emptyset$,
$$
\meas_2(U_{k(i,j)}^{*}) \le Q_{2} \min (\meas_2(U_{i}),\meas_2(U_{j})).
$$
\end{enumerate}
\end{definition}

When $A=A^{\#}=X$, we say that $(U_{i},U_{i}^{*},U_{i}^{\#})_{i \in I}$ is a good covering of $(X,\dist)$ with respect to $(\meas_1,\meas_2)$.

For the sake of clarity, we call condition \ref{3}.~the \textit{overlapping condition}, condition \ref{4}.~the \textit{embracing condition} and condition \ref{5}.~the \textit{measure control condition} of the good covering. Note that in \cite{Minerbe} the measure control condition was required also for $\meas_1$ though never used in the proofs.

From now on, we consider two numbers $p, q \in[1,+\infty)$ and two Borel sets $A\subset A^\# \subset X$. We assume that a good covering $(U_{i},U_{i}^{*},U_{i}^{\#})_{i \in I}$ of $(A,A^{\#})$ with respect to $(\meas_1,\meas_2)$ exists.

Let us explain how to define from $(U_{i},U_{i}^{*},U_{i}^{\#})_{i \in I}$ a canonical weighted graph $(\cV,\cE,\nu)$, where $\cV$ is the set of vertices of the graph, $\cE$ is the set of edges, and $\nu$ is a weight on the graph (i.e. a function $\nu:\cV \sqcup \cE \rightarrow \R$). We define $\cV$ by associating to each $U_{i}$ a vertex $i$ (informally, we put a point $i$ on each $U_i$). Then we set $\cE := \{ (i,j) \in \cV\times\cV : i \neq j \, \, \text{and} \, \, \ov{U_{i}} \cap \ov{U_{j}} \neq \emptyset \}$. Finally we weight the vertices of the graph by setting $\nu(i):=\meas_2(U_{i})$ for every $i \in \cV$  and the edges by setting $\nu(i,j):=\max(\nu(i),\nu(j))$ for every $(i,j) \in \cE$. 

The patching theorem (Theorem \ref{th:patching}) states that if some local inequalities are true on the pieces of the good covering and if a discrete inequality holds on the associated canonical weighted graph, then the local inequalities can be patched into a global one. Let us give the precise definitions.

\begin{definition}[Local continuous $L^{q,p}$-Sobolev-Neumann inequalities]\label{def:localSobolevNeumann}
We say that the good covering $(U_{i},U_{i}^{*},U_{i}^{\#})_{i \in I}$ satisfies local continuous $L^{q,p}$-Sobolev-Neumann inequalities if there exists a constant $S_{c} > 0$ such that for all $i \in I$,
\begin{equation}\label{eq:1}
\left(  \int_{U_{i}} |u - \{u\}_{U_i}|^{q} \di \meas_2 \right)^{\frac{1}{q}} \le S_{c} \left( \int_{U_{i}^{*}} g^p \di \meas_1 \right)^{\frac{1}{p}}
\end{equation}
for all $u \in L^1(U_i,\meas_2)$ and all upper gradients $g \in L^p(U_i^*,\meas_1)$, and
\begin{equation}\label{eq:2}
\left(  \int_{U_{i}^{*}} |u - \{u\}_{U_i^*}|^{q} \di \meas_2\right)^{\frac{1}{q}} \le S_{c}\left(  \int_{U_{i}^{\#}} g^p \di \meas_1 \right)^{\frac{1}{p}}
\end{equation}
for all $u \in L^1(U_i^*,\meas_2)$ and all upper gradients $g \in L^p(U_i^\#,\meas_1)$.
\end{definition}

\begin{definition}[Discrete $L^q$-Poincaré inequality]\label{def:discreteSobolevDirichlet}
We say that the weighted graph $(\cV,\cE,\nu)$ satisfies a discrete $L^q$-Poincaré inequality if there exists $S_{d}>0$ such that:
\begin{equation}\label{eq:discretePoincaré}
\left( \sum_{i \in \cV} |f(i)|^q \nu(i) \right)^{\frac{1}{q}} \le S_{d} \left( \sum_{\{ i,j \} \in \cE} |f(i) - f(j)|^q \nu(i,j) \right)^{\frac{1}{q}}  \qquad \forall f \in L^{q}(\cV,\nu).
\end{equation}
\end{definition}

\begin{remark}
Here we differ a bit from Minerbe's terminology. Indeed, in \cite{Minerbe}, the following discrete $L^q$ Sobolev-Dirichlet inequalities of order $k$ were introduced for any $k \in (1,+\infty]$ and any $q \in [1,k)$:
\[
\left( \sum_{i \in \cV} |f(i)|^{\frac{qk}{k-q}} \nu(i) \right)^{\frac{k-q}{qk}} \le S_{d} \left( \sum_{\{ i,j \} \in \cE} |f(i) - f(j)|^q \nu(i,j) \right)^{\frac{1}{q}}\qquad \forall f \in L^q(\cV,\nu).
\]
In the present paper we only need the case $k=+\infty$, in which we recover \eqref{eq:discretePoincaré}: here is why we have chosen the terminology ``Poincaré'' which seems, in our setting, more appropriate.
\end{remark}




We are now in a position to state the patching theorem.

\begin{theorem}[Patching theorem]\label{th:patching}
Let $(X,\dist)$ be a metric space equipped with two Borel measures $\meas_1$ and $\meas_2$, both finite and nonzero on balls with finite and nonzero radius, such that $\supp(\meas_1) = \supp(\meas_2)=X$. Let $A\subset A^\# \subset X$ be two Borel sets, and $p,q \in [1,+\infty)$ be such that $q\ge p$. Assume that $(A,A^{\#})$ admits a good covering $(U_i,U_i^*,U_i^\#)$ with respect to $(\meas_1,\meas_2)$ which satisfies the local $L^{q,p}$-Sobolev-Neumann inequalities \eqref{eq:1} and \eqref{eq:2} and whose associated weighted graph $(\cV,\cE,\nu)$ satisfies the discrete $L^q$-Poincaré inequality \eqref{eq:discretePoincaré}. Then there exists a constant $C=C(p,q,Q_1,Q_2,S_c,S_d) > 0$ such that for any function $u \in C_c(A^\#)$ and any upper gradient $g \in L^p(A^\#,\meas_1)$ of $u$, $$ \left( \int_A |u|^{q} \di \meas_2 \right)^{\frac{1}{q}} \le C \left( \int_{A^{\#}} g^{p} \di \meas_1  \right)^{\frac{1}{p}}.$$
\end{theorem}

Although the proof of Theorem \ref{th:patching} is a straightforward adaptation of \cite[Th.~1.8]{Minerbe}, we provide it for the reader's convenience.

\begin{proof}
Let us consider $u \in C_c(A^{\#})$. Then
$$\int_A |u|^q \di\meas_2 \le \sum_{i \in \cV} \int_{U_i} |u|^q \di \meas_2.$$ From convexity of the function $t \mapsto |t|^q$, we deduce $|u|^q \le 2^{q-1}(|u-\{u\}_{U_i}|^q+|\{u\}_{U_i}|^q)$ $\meas_2$-a.e.~on each $U_i$, and then
\begin{equation}\label{eq:11111}
\int_A |u|^q \di\meas_2 \le 2^{q-1} \sum_{i \in \cV} \int_{U_i} |u-\{u\}_{U_i}|^q \di \meas_2 + 2^{q-1} \sum_{i \in \cV} |\{u\}_{U_i}|^q \nu(i).
\end{equation}
From \eqref{eq:1} and the fact that $\sum_jx_j^{q/p} \le (\sum_j x_j)^{q/p}$ for any finite family of non-negative numbers $\{x_j\}$ (since $q\ge p$), we get
\begin{align}\label{eq:22222}
\sum_{i \in \cV} \int_{U_i} |u-\{u\}_{U_i}|^q \di \meas_2 & \le S_c^{q/p} \left( \sum_{i \in \cV} \int_{U_i^*} g^p \di \meas_1 \right)^{q/p} \nonumber \\
& \le S_c^{q/p} Q_1^{q/p} \left(\int_{A^{\#}}g^p \di \meas_1\right)^{q/p},
\end{align}
this last inequality being a direct consequence of the overlapping condition \ref{3}. Now the discrete $L^q$-Poincaré inequality \eqref{eq:discretePoincaré} implies
\begin{equation}\label{eq:33333}
\sum_{i \in \cV} |\{u\}_{U_i}|^q \nu(i) \le S_d \sum_{(i,j)\in \cE} |\{u\}_{U_i} - \{u\}_{U_j}|^q \nu(i,j).
\end{equation}
For any $(i,j) \in \cE$, a double application of Hölder's inequality yields to
$$
|\{u\}_{U_i} - \{u\}_{U_j}|^q \nu(i,j) \le \frac{\nu(i,j)}{\meas_2(U_i) \meas_2(U_j)} \int_{U_i} \int_{U_j} |u(x) - u(y)|^q \di \meas_2(x) \di \meas_2(y),
$$
and as the measure control condition \ref{5}.~ensures $\nu(i,j)=\max(\meas_2(U_i),\meas_2(U_j))\le Q_2 \meas_2(U_{k(i,j)}^*)$, the embracing condition \ref{4}.~implies
$$
|\{u\}_{U_i} - \{u\}_{U_j}|^q \nu(i,j)\le \frac{Q_2}{\meas_2(U_{k(i,j)}^*)} \int_{U_{k(i,j)}^*} \int_{U_{k(i,j)}^*} |u(x) - u(y)|^q \di \meas_2(x) \di \meas_2(y)
$$
and then
$$
|\{u\}_{U_i} - \{u\}_{U_j}|^q \nu(i,j)\le Q_2 2^q  \int_{U_{k(i,j)}^*} |u - \{u\}_{U_{k(i,j)}^*}|^q \di \meas_2
$$
where we have used again the convexity of $t \mapsto |t|^q$. Summing over $(i,j) \in \cE$, we get
\begin{equation}\label{eq:44444}
\sum_{(i,j)\in \cE} |\{u\}_{U_i} - \{u\}_{U_j}|^q \nu(i,j) \le Q_2 2^q \sum_{(i,j)\in \cE} \int_{U_{k(i,j)}^*}|u - \{u\}_{U_{k(i,j)}^*}|^q \di \meas_2.
\end{equation}
Then \eqref{eq:2} yields to
\begin{equation}\label{eq:55555}
\sum_{(i,j)\in \cE} |\{u\}_{U_i} - \{u\}_{U_j}|^q \nu(i,j) \le Q_2 2^q  S_c^{q/p} \left( \sum_{(i,j)\in \cE} \int_{U_{k(i,j)}^\#} g^p \di \meas_1 \right)^{q/p}.
\end{equation}
Finally, a simple counting argument shows that 
\begin{equation}\label{eq:66666}
\sum_{(i,j)\in \cE} \int_{U_{k(i,j)}^\#} g^p \di \meas_1 \le Q_1^3 \int_{A^\#} g^p \di \meas.
\end{equation}
The result follows from combining \eqref{eq:11111}, \eqref{eq:22222}, \eqref{eq:33333}, \eqref{eq:44444}, \eqref{eq:55555} and \eqref{eq:66666}.
\end{proof} 

A similar statement holds if we replace the discrete $L^q$-Poincaré inequality by a discrete ``$L^q$-Poincaré-Neumann'' version:
\begin{equation}\label{eq:discretePoincaréNeumann}
\left( \sum_{i \in \cV} |f(i) - \nu(f)|^q \nu(i) \right)^{\frac{1}{q}} \le S_{d} \left( \sum_{\{ i,j \} \in \cE} |f(i) - f(j)|^q \nu(i,j) \right)^{\frac{1}{q}}
\end{equation}
for all compactly supported $f:\cV \to \R$, where $\nu(f) = \left(  \sum_{i \, : \, f(i) \neq 0} \nu(i)\right)^{-1} \sum_{i} f(i)\nu(i)$. The terminology ``Poincaré-Neumann'' comes from the mean-value in the left-hand side of \eqref{eq:discretePoincaréNeumann} and the analogy with the local Poincaré inequality used in the study of the Laplacian on bounded Euclidean domains with Neumann boundary conditions, see \cite[Sect.~1.5.2]{Saloff-Coste}.

\begin{theorem}[Patching theorem - Neumann version]\label{th:patching2}
Let $(X,\dist)$ be a metric space equipped with two Borel measures $\meas_1$ and $\meas_2$, both finite and nonzero on balls with finite and nonzero radius, such that $\supp(\meas_1) = \supp(\meas_2)=X$. Let $A\subset A^\# \subset X$ be two Borel sets such that $0<\meas(A) < +\infty$ and $p,q \in [1,+\infty)$ such that $q\ge p$. Assume that $(A,A^{\#})$ admits a good covering $(U_i,U_i^*,U_i^\#)$ with respect to $(\meas_1,\meas_2)$ which satisfies the local $L^{q,p}$-Sobolev-Neumann inequalities \eqref{eq:1} and \eqref{eq:2} and whose associated weighted graph $(\cV,\cE,\nu)$ satisfies the discrete $L^q$-Poincaré-Neumann inequality \eqref{eq:discretePoincaréNeumann}. Then there exists a constant $C=C(p,q,Q_1,Q_2,S_c,S_d) > 0$ such that for any $u \in C_c(A^\#)$ and any upper gradient $g \in L^p(A^\#,\meas_1)$, \[ \left( \int_{A} |u - \{ u \}_A|^{q} \di \meas_2 \right)^{\frac{1}{q}} \le C \left( \int_{A^{\#}} g^{p} \di \meas_1 \right)^{\frac{1}{p}}.\]
\end{theorem}

The proof of Theorem \ref{th:patching2} is similar to the proof of Theorem \ref{th:patching} and writes exactly as \cite[Th.~1.10]{Minerbe} with upper gradients instead of norms of gradients, so we skip it.\\

\section{Proof of the main result}

In this section, we prove Theorem \ref{th:weightedSobolev} after a few preliminary results.

As already pointed out in \cite{Minerbe}, the local continuous $L^{2^*,2}$-Sobolev-Neumann inequalities on Riemannian manifolds (where $2^*=2n/(n-2)$ and $n$ is the dimension of the manifold) can be derived from the doubling condition and the uniform strong local $L^2$-Poincaré inequality which are both implied by non-negativity of the Ricci curvature. However, the discrete $L^{2^*}$-Poincaré inequality requires an additional reverse doubling condition which is an immediate consequence of the growth condition (\ref{eq:growthcondition}), as shown in the next lemma.

\begin{lemma}\label{lemma}
Let $(Y,\dist_Y, \meas_Y)$ be a metric measure space such that
\begin{equation}\label{eq:lemma}
0 < \Theta_{inf}:= \liminf_{r\to +\infty} \frac{\meas_Y(B_r(y_o))}{r^\alpha} \le \Theta_{sup}:= \limsup_{r\to +\infty} \frac{\meas_Y(B_r(y_o))}{r^\alpha} < +\infty
\end{equation}
for some $y_o \in Y$ and $\alpha>0$. Then there exists $A>0$ and $C_{RD}=C_{RD}(\Theta_{inf},\Theta_{sup})>0$ such that
\begin{equation}\label{eq:reverselemma}
\frac{\meas_Y(B_R(y_o))}{\meas_Y(B_r(y_o))} \ge C_{RD} \left( \frac{R}{r} \right)^{\alpha} \qquad \forall \, A<r\le R.
\end{equation}
\end{lemma}

\begin{proof}
The growth condition (\ref{eq:lemma}) implies the existence of $A>0$ such that for any $R \ge r > A$, $\Theta_{inf}/2 \le r^{-\alpha} \meas_Y(B_r(y_o)) \le 2 \Theta_{sup}$ and $R^{-\alpha}\meas_Y(B_R(y_o)) \ge \Theta_{inf}/2$, whence \eqref{eq:reverselemma} with $C_{RD}=\Theta_{inf}/(4\Theta_{sup})$.
\end{proof}

\begin{remark}
Note that the doubling condition \eqref{doubling} easily implies \eqref{eq:reverselemma}: see for instance \cite[p.9]{Grigor'YanHuLau} for a proof giving $C_{RD}=(1+C_{D}^{-4})^{-1}$ and $\alpha = \log_{2}(1+C_{D}^{-4})$. But in this case, $\alpha>1$ if and only if $C_{D}<1$ which is impossible. So we emphasize that in our context, in which we want the segment $(1,\alpha)$ to be non-empty, doubling and reverse doubling must be thought as complementary hypotheses.
\end{remark}

The next result, a strong local $L^p$-Sobolev inequality for $\CD(0,N)$ spaces, is an important technical tool for our purposes. In the context of Riemannian manifolds, it was proved by Maheux and Saloff-Coste \cite{MaheuxSaloffCoste}.


\begin{lemma}\label{lem:averaged}
Let $(Y,\dist_Y,\meas_Y)$ be a $\CD(0,N)$ space. Then for any $p \in [1,N)$ there exists $C=C(N,p)>0$ such that for any $u \in C(Y)$, any upper gradient $g \in L^1_{loc}(Y,\meas_Y)$, and any ball $B$ with arbitrary radius $r>0$,
\begin{equation}\label{eq:averaged}
\left( \int_B |u - u_B|^{p^*} \di \meas_Y \right)^{\frac{1}{p^*}} \le C \frac{r}{\meas_Y(B)^{1/N}} \left( \int_B g^p \di \meas_Y \right)^{\frac{1}{p}},
\end{equation}
where $p^*=Np/(N-p)$.
\end{lemma}

\begin{proof}
Let $u$ be a continuous function on $Y$, $g \in L^1_{loc}(Y,\meas_Y)$ be an upper gradient of $u$, $B$ be a ball with arbitrary radius $r>0$, and $p\in[1,N)$. In this proof $u_B$ stands for $\meas_Y(B)^{-1} \int_B u \di \meas_Y$. Thanks to Hölder's inequality and the doubling property, Proposition \ref{prop:Rajala} implies
\[
\fint_{B}|u-u_{B}| \di\meas_Y \le 2^{N+2} r \left( \fint_{2B} g^p \di\meas_Y \right)^{1/p}.
\]
Let $x_0, x_1 \in Y$ and $r_0, r_1>0$ be such that $x_1 \in B_{r_0}(x_0)$ and $r_1 \le r_0$. Then
$$
\frac{\meas_Y(B_{r_1}(x_1))}{\meas_Y(B_{r_0}(x_0))} \ge \frac{\meas_Y(B_{r_1}(x_1))}{\meas_Y(B_{r_0+\dist_Y(x_0,x_1)}(x_1))} \ge 2^{-N} \left( \frac{r_1}{r_0 + \dist_Y(x_0,x_1)} \right)^N \ge 2^{-2N} \left( \frac{r_1}{r_0} \right)^N
$$
by the doubling condition. Moreover, we know from Proposition \ref{truncation} that $(u,g)$ satisfies the truncation property, so that \cite[Th.~5.1, 1.]{HajlaszKoskela} applies and gives
\[
\left( \fint_{B} |u - u_{B}|^{p^{*}} \di\meas_Y \right)^{1/p^{*}} \le \tilde{C} r \left( \fint_{10B} g^p \di\meas_Y \right)^{1/p}
\]
where $\tilde{C}$ depends only on $p$ and the doubling and Poincaré constants of $(Y,\dist_Y,\meas_Y)$ which depend only on $N$.
As $(Y,\dist_Y,\meas_Y)$ is a $CD(0,N)$ space, the metric structure $(Y,\dist_Y)$ is proper and geodesic, so it follows from \cite[Cor.~9.5]{HajlaszKoskela} that all the balls in $Y$ are John domains with a universal constant $C_J>0$. Then \cite[Th.~9.7]{HajlaszKoskela} applies and yields to the result since $1/p^* - 1/p = 1/N$.
\end{proof}


Finally, let us state a result whose proof - omitted here - can be deduced from \cite[Prop.~2.8]{Minerbe} by using Proposition \ref{prop:Rajala}. Note that even if Proposition \ref{prop:Rajala} provides only a weak inequality, one can harmlessly substitute it to the strong one used in the proof of \cite[Prop.~2.8]{Minerbe}, because it is applied there to a function $f$ which is Lipschitz on a ball $B$ and extended by $0$ outside of $B$. Note also that Proposition \ref{prop:Rajala} being a $L^1$-Poincaré inequality, we can assume $\alpha >1$ (a $L^2$-Poincaré inequality would have only permit $\alpha >2$).  
\begin{prop}\label{prop:RCA}
Let $(Y,\dist_Y,\meas_Y)$ be a $\CD(0,N)$ space  satisfying the growth condition \eqref{eq:lemma} with $\alpha>1$. Then there exists $\kappa_{0}=\kappa_{0}(N,\alpha)>1$ such that for any $R>0$ such that $S_R(y_o)$ is non-empty, for any couple of points $(x,x') \in S_R(y_o)^2$, there exists a rectifiable curve from $x$ to $x'$ that remains inside $B_R(y_o) \backslash B_{\kappa_{0}^{-1}R}(y_o)$.
\end{prop}


Let us prove now Theorem \ref{th:weightedSobolev}. Let $(X,\dist,\meas)$ be a non-compact $\CD(0,N)$ space with $N \ge 3$ satisfying the growth condition (\ref{eq:growthcondition}) with parameter $\eta \in (1,N]$, and $p \in [1,\eta)$. We recall that $\mu$ is the measure absolutely continuous with respect to $\meas$ with density $w_o = V(o,\dist(o,\cdot))^{p/(N-p)} \dist(o,\cdot)^{-Np/(N-p)}$, and that $p^*=Np/(N-p)$. Note that Lemma \ref{lemma} applied to $(X,\dist,\meas)$, assuming with no loss of generality that $A=1$, implies:
\begin{equation}\label{eq:reverse}
\frac{V(o,R)}{V(o,r)} \ge C_{RD} \left(\frac{R}{r} \right)^\eta \qquad \forall \, 1 < r < R.\\
\end{equation}

\hfill

\noindent \underline{STEP 1:} The good covering.\\

Let us briefly explain how to construct a good covering on $(X,\dist,\meas)$, referring to \cite[Section 2.3.1]{Minerbe} for additional details. Define $\kappa$ as the square-root of the constant $\kappa_{0}$ given by Proposition \ref{prop:RCA}. Then for any $R>0$, two connected components $X_1$ and $X_2$ of $B_{\kappa R}(o)\backslash B_R(o)$ are always contained in one component of $B_{\kappa R}(o)\backslash B_{\kappa^{-1} R}(o)$: otherwise, linking $x \in \overline{X_1}\cap S_{\kappa R}(o)$ and $x' \in \overline{X_2} \cap S_{\kappa R}(o)$ by a curve remaining inside $B_{\kappa R}(o)\backslash B_{\kappa^{-1} R}(o)$ would not be possible.


Every point in a complete geodesic metric space of infinite diameter is the origin of some geodesic ray: see e.g.~\cite[Prop.~10.1.1]{Papadopoulos}. Therefore, there exists a geodesic ray $\gamma$ starting from $o$. For any $i \in \N$, let us write $A_{i}=B_{\kappa^{i}}(o) \backslash B_{\kappa^{i-1}}(o)$ and denote by $(U'_{i,a})_{0\le a \le h_{i}'}$ the connected components of $A_{i}$, $U_{i,0}'$ being set as the one intersecting $\gamma$. The next simple result was used without a proof in \cite{Minerbe}.

\begin{claim}\label{lem2}
There exists a constant $h=h(N,\kappa) < \infty$ such that $\sup_i h_{i}' \le h$.
\end{claim}
\begin{proof}
Take $i \in \N$. For every $0 \le a \le h_{i}'$, pick $x_{a}$ in $U_{i,a} \cap S_{(\kappa^{i}+\kappa^{i-1})/2}(o) $. As the balls $(B_a:=B_{(\kappa^i-\kappa^{i-1})/4}(x_a))_{0 \le a \le h_i'}$ are disjoint and all included in $B_{\kappa^{i}}(o)$, we have
\[ h_{i}' \min_{0 \le a \le h_{i}'} 
\meas(B_a) \le \sum_{0 \le a \le h_{i}'} \meas(B_a) \le V(o,\kappa^{i}). \] With no loss of generality, we can assume that $\min\limits_{0 \le a \le h_{i}'} \meas(B_a)  = \meas(B_0)$.  Notice that $d(o,x_{0}) \le \kappa^{i}$. Then
$$h_{i}' \le \frac{V(o,\kappa^{i})}{\meas(B_0)} \le \frac{V(x_{0},\kappa^{i} + \dist(o,x_{0}))}{\meas(B_0)} \le \left( \frac{8\kappa^{i}}{\kappa^{i}-\kappa^{i-1}}\right)^{N} $$ by the doubling condition.  This yields to the result with $h:=\left( \frac{8\kappa}{\kappa-1} \right)^{N}$.
\end{proof}

Define then the covering $(U_{i,a}',U_{i,a}'^{*},U_{i,a}'^{\#})_{i \in \N,0 \le a \le h_{i}'}$ where $U_{i,a}'^{*}$ is by definition the union of the sets $U_{j,b}'$ such that $\overline{U_{j,b}'} \cap \overline{U_{i,a}'} \neq \emptyset$, and $U_{i,a}'^{\#}$ is by definition the union of the sets $U_{j,b}'^{*}$ such that $\overline{U_{j,b}'^{*}} \cap \overline{U_{i,a}'^{*}} \neq \emptyset$. Note that $(U_{i,a}',U_{i,a}'^{*},U_{i,a}'^{\#})_{i \in \N,0 \le a \le h_{i}'}$ is not necessarily a good covering, as pieces $U_{i,a}'$ might be arbitrary small compared to their neighbors: in this case, the measure control condition \ref{5}.~would not be true. So whenever $\overline{U_{i+1,a}'} \cap S_{\kappa^{i+1}}(o) = \emptyset$ (this condition being satisfied by all ``small'' pieces), we set $U_{i,a} := U_{i+1,a}' \cup U_{i,a'}'$ where $a'$ is the integer such that $\overline{U_{i+1,a}'} \cap \overline{U_{i,a'}} \neq \emptyset$; otherwise we set $U_{i+1,a} := U_{i+1,a}'$.

\begin{figure}[h]
\centering
\def\svgwidth{\textwidth}
\scalebox{.5}{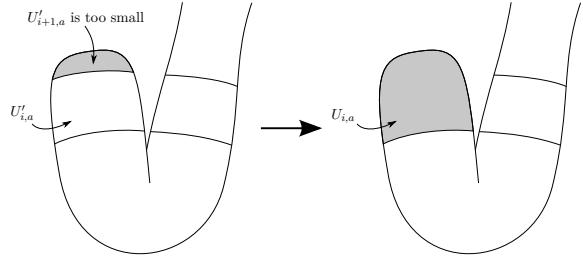}
\caption{for simplicity assume $a'=a$; if $U_{i+1,a}' \cap S_{\kappa^{i+1}}(o) = \emptyset$, then we glue the small piece $U_{i+1,a}'$ to the adjacent piece $U_{i,a}'$ to form $U_{i,a}$}
\end{figure}

We define $U_{i,a}^{*}$ and $U_{i,a}^{\#}$ in a similar way from $U_{i,a}'^{*}$ and $U_{i,a}'^{\#}$ respectively. Using the doubling condition, one can easily show that $(U_{i,a},U_{i,a}^{*},U_{i,a}^{\#})_{i \in \N, 0 \le a \le h_{i}} $ is a good covering of $(X,\dist)$ with respect to $(\mu,\meas)$, with constants $Q_1$ and $Q_2$ depending only on $N$.

\hfill

\noindent \underline{STEP 2:} The discrete $L^{p^*}$-Poincaré inequality.\\

Let $(\cV,\cE,\nu)$ be the weighted graph obtained from $(U_{i,a},U_{i,a}^{*},U_{i,a}^{\#})_{i \in \N, 0 \le a \le h_{i}} $. Define the degree $\deg(i,a)$ of a vertex $(i,a)$ as the number of vertices $(j,b)$ such that $\overline{U_{i,a}} \cap \overline{U_{j,b}} = \emptyset$. As a consequence of Claim \ref{lem2}, $\sup \{\deg(i,a) : (i,a) \in \cV\} \le 2h$. Moreover:

\begin{claim}
There exists $C \ge 1$ such that $C^{-1} \le \nu(j,b)/\nu(i,a) \le C$ for any $(i,a),(j,b) \in \cE$.
\end{claim}

\begin{proof}
Take $(i,a),(j,b) \in \cE$. With no loss of generality we can assume $j=i+1$. Take $x \in U_{i,a} \cap S_{(\kappa^{i}+\kappa^{i-1})/2}(o)$  and set $r=(\kappa^{i}-\kappa^{i-1})/4$, $R=2\kappa^{i+1}$, so that $B_r(x) \subset U_{i,a}$ and $U_{i+1,b} \subset B_R(x)$. Then the doubling condition implies
\[
\nu(i+1,b)\le \meas_2(B_R(x)) \le C_D (R/r)^{\log_2 C_D} \meas_2(B_r(x)) \le \bar{C} \nu(i,a)
\]
where $\bar{C}=C_D (8 \kappa^2/(\kappa-1))^{\log_2(C_D)} \ge 1$. A similar reasoning starting from $x \in U_{i+1,b} \cap S_{(\kappa^{i+1}+\kappa^{i})/2}(o)$ provides the existence of $C'\ge 1$ such that $\nu(i,a) \le C' \nu(i+1,b)$. Set $C=\max(\bar{C},C')$ to conclude.
\end{proof}

We are now in a position to apply \cite[Prop.~1.12]{Minerbe} which ensures that the discrete $L^1$-Poincaré inequality implies the $L^q$ one for any given $q\ge1$. But the discrete $L^1$-Poincaré inequality is equivalent to the isoperimetric inequality (\cite[Prop.~1.14]{Minerbe}): there exists a constant $\cI>0$ such that for any $\Omega \subset \cV$ with finite measure, $$\frac{\nu(\Omega)}{\nu(\partial \Omega)} \le \cI$$ where $\partial \Omega := \{ (i,a),(j,b)\in \cE : (i,a)\in\Omega, (j,b)\notin \Omega \}$. The only ingredients to prove this isoperimetric inequality are the doubling and reverse doubling conditions, see Section 2.3.3 in \cite{Minerbe}. Then the discrete $L^q$-Poincaré inequality holds for any $q \ge 1$, with a constant $S_d$ depending only on $q$, $\eta$, $\Theta_{inf}$, $\Theta_{sup}$ and on the doubling and Poincaré constants of $(X,\dist,\meas)$, i.e. on $N$. In case $q=p^*$, we have $S_d=S_d(N,\eta,p,\Theta_{inf},\Theta_{sup})$.\\

\noindent \underline{STEP 3:} The local continuous $L^{p^*,p}$-Sobolev-Neumann inequalities.\\

Let us explain how to get the local continuous $L^{p^*,p}$-Sobolev-Neumann inequalities. We start by deriving from the strong local $L^p$-Sobolev inequality \eqref{eq:averaged} a $L^p$-Sobolev-type inequality on connected Borel subsets of annuli.

\begin{claim}\label{claim2}
Let $R>0$ and $\alpha>1$. Let $A$ be a connected Borel subset of $B_{\alpha R}(o) \backslash B_R(o)$. For $0 < \delta < 1$, denote by $[A]_{\delta}$ the $\delta$-neighborhood of $A$, i.e. $[A]_{\delta}=\bigcup_{x \in A} B_\delta (x)$. Then there exists a constant $C=C(N,\delta,\alpha,p)>0$ such that for any function $u \in C(X)$ and any upper gradient $g \in L^p([A]_{\delta},\meas)$ of $u$,
$$ \left( \int_{A}|u-u_{A}|^{p^{*}} \di \meas \right)^{1/p^*} \le C \frac{R^p}{V(o,R)^{p/N}} \left( \int_{[A]_{\delta}} g^p \di \meas \right)^{1/p}.$$
\end{claim}

\begin{proof}
Define $s=\delta R$ and choose an $s$-lattice of $A$ (i.e.~a maximal set of points whose distance between two of them is at least $s$) $(x_{j})_{j \in J}$. Set $V_{i}=B(x_{i},s)$ and $V_{i}^{*}=V_{i}^{\#}=B(x_{i},3s)$. Using the doubling condition, there is no difficulty in proving that $(V_{i},V_{i}^{*},V_{i}^{\#})$ is a good covering of $(A,[A]_\delta)$ with respect to $(\meas,\meas)$. A discrete $L^{p^*}$-Poincaré inequality holds on the associated weighted graph, as one can easily check following the lines of \cite[Lem.~2.10]{Minerbe}. The local continuous $L^{p^*,p}$-Sobolev-Neumann inequalities stem from the proof of \cite[Lem.~2.11]{Minerbe}, where we replace (14) there by \eqref{eq:averaged}. Then Theorem \ref{th:patching2} gives the result.
\end{proof}

Let us prove that Claim \ref{claim2} implies the local continuous $L^{p^*,p}$-Sobolev-Neumann inequalities with a constant $S_c$ depending only on $N$, $\eta$ and $p$. Take a piece of the good covering $U_{i,a}$. Choose $\delta=(1-\kappa^{-1})/2$ so that $[U_{i,a}]_{\delta} \subset U_{i,a}^{*}$. Take a function $u \in C(X)$ and an upper gradient $g\in L^p([U_{i,a}]_\delta, \meas)$ of $u$. Since $|u-\langle u \rangle_{U_{i,a}}| \le |u - c| + |c - \langle u \rangle_{U_{i,a}}|$ for any $c\in\R$, convexity of $t\mapsto |t|^{p^*}$ and Hölder's inequality imply
\begin{align*}
\int_{U_{i,a}}|u-\langle u \rangle_{U_{i,a}}|^{p^*} \di \mu & \le 2^{p^*-1}\int_{U_{i,a}}(|u - c|^{p^*} + |c - \langle u \rangle_{U_{i,a}}|^{p^*}) \di \mu\\
& \le 2^{p^*} \inf_{c \in \R} \int_{U_{i,a}}|u-c|^{p^*} \di\mu \le 2^{p^*} \int_{U_{i,a}}|u-u_{U_{i,a}}|^{p^*} w_o \di \meas.
\end{align*}
As $w_o$ is a radial function, we can set $\bar{w}_o (r) := w_o(x)$ for any $r>0$ and any $x \in X$ such that $\dist(o,x)=r$.  Note that by the Bishop-Gromov theorem, $\bar{w}_o$ is a decreasing function, so
\[
\int_{U_{i,a}}|u-\langle u \rangle_{U_{i,a}}|^{p^*} \di \mu \le 2^{p^*} \bar{w}_o(\kappa^{i-1}) \int_{U_{i,a}} |u - u_{U_{i,a}}|^{p^*} \di \meas.
\]
Applying Claim \ref{claim2} with $A = U_{i,a}$, $R=\kappa^{i-1}$ and $\alpha=\kappa^2$ yields to
\begin{align*}
\int_{U_{i,a}}|u-\langle u \rangle_{U_{i,a}}|^{p^*} \di \mu &\le  2^{p^*}  \bar{w}_o(\kappa^{i-1}) \, C  \, \frac{\kappa^{p p^{*}(i-1)}}{V(o,\kappa^{i-1})^{p p^*/N}} \left( \int_{U_{i,a}^{*}} g^p \di \meas \right)^{p^*/p} \\
& \le C \left( \int_{U_{i,a}^{*}} g^p \di \meas \right)^{p^*/p} 
\end{align*}
where we used the same letter $C$ to denote different constants depending only on $N$, $p$, and $\kappa$. As $\kappa$ depends only on $N$, $\eta$ and $p$, we get the result.


An analogous argument implies the inequalities between levels 2 and 3.\\

\noindent \underline{STEP 4:} Conclusion.\\

Apply Theorem \ref{th:patching} to get the result.

\section{Weighted Nash inequality and bound of the corresponding heat kernel}

In this section, we deduce from Theorem \ref{th:weightedSobolev} a weighted Nash inequality. We use this result in the context of $\RCD(0,N)$ spaces to get a uniform bound on a corresponding weighted heat kernel.



\begin{theorem}[Weighted Nash inequality]\label{th:weightedNash}
Let $(X,\dist,\meas)$ be a $\CD(0,N)$ space with $N> 2$ satisfying $\eqref{eq:growthcondition}$ with $\eta>2$. Then there exists a constant $C_{Na}=C_{Na}(N,\eta,\Theta_{inf},\Theta_{sup})>0$ such that: 
\[
\|u\|_{L^2(X,\mu)}^{2+\frac{4}{N}} \le C_{Na} \|u\|_{L^1(X,\mu)}^{\frac{4}{N}} \Ch(u) \qquad \forall \, u \in L^{1}(X,\mu) \cap H^{1,2}(X,\dist,\meas),
\]
where $\mu\ll \meas$ has density $w_o= V(o,\dist(o,\cdot))^{2/(N-2)} \dist(o,\cdot)^{-2N/(N-2)}$.
\end{theorem}

To prove this theorem, we need a standard lemma which states that the relaxation procedure defining $\Ch$ can be performed with slopes of Lipschitz functions with bounded support (we write $\Lip_{bs}(X)$ in the sequel for the space of such functions) instead of upper gradients of $L^2$-functions. We omit the proof for brevity and refer to the paragraph after Propositon 4.2 in \cite{AmbrosioColomboDiMarino} for a discussion on this result. Note that here and until the end of this section we write $L^p(\meas)$, $L^p(\mu)$ instead of $L^p(X,\meas)$, $L^p(X,\mu)$ respectively for any $1 \le p \le +\infty$.

\begin{lemma}\label{lem:standardlemma}
Let $(X,\dist,\meas)$ be a complete and separable metric measure space, and $u \in H^{1,2}(X,\dist,\meas)$. Then $$\Ch(u)=\inf \left\{ \liminf\limits_{n \to \infty} \int_{X} |\nabla u_{n}|^2 \di \meas : (u_{n})_{n} \subset \Lip_{bs}(X), \|u_{n} - u\|_{L^2(\meas)}\rightarrow 0 \right\}. $$
In particular, for any $u \in H^{1,2}(X,\dist,\meas)$, there exists a sequence $(u_n)_n \subset \Lip_{bs}(X)$ such that $\|u-u_n\|_{L^2(\meas)} \to 0$ and $\||\nabla u_n| \|_{L^2(\meas)}^2 \to \Ch(u)$ when $n \to +\infty$.
\end{lemma}

We are now in a position to prove Theorem \ref{th:weightedNash}.

\begin{proof}
By the previous lemma it is sufficient to prove the result for $u \in \Lip_{bs}(X)$. By Hölder's inequality, $$\|u\|_{L^2(\mu)} \le \|u\|_{L^1(\mu)}^{\theta}\|u\|_{L^{2^*}(\mu)}^{1-\theta}$$ where $\frac{1}{2} = \frac{\theta}{1} + \frac{1-\theta}{2^*}$ i.e. $\theta = \frac{2}{N+2}$. Then by Theorem \ref{th:weightedSobolev} applied in the case $p = 2 < \eta$, $$\|u\|_{L^{2}(\mu)} \le C \|u\|_{L^1(\mu)}^{\frac{2}{N+2}} \||\nabla u|\|_{L^{2}(\meas)}^{\frac{N}{N+2}}.$$
It follows from the identification between slopes and minimal relaxed gradients established in \cite[Theorem 5.1]{Cheeger} that $\Ch(u) = \||\nabla u|\|_{L^{2}(\meas)}^2$, so the result follows by raising the previous inequality to the power $2(N+2)/N$.
\end{proof}

Let us consider now a $\RCD(0,N)$ space $(X,\dist,\meas)$ satisfying the growth condition \eqref{eq:growthcondition} for some $\eta>2$ and the uniform local $N$-Ahlfors regularity property :
\begin{equation}\label{eq:Ahlfors}
C_o^{-1} \le \frac{V(x,r)}{r^N} \le C_o \qquad \forall x \in X, \, \, \, \forall \, 0 < r < r_o
\end{equation}
for some $C_o>1$ and $r_o>0$. Such spaces are called \textit{weakly non-collapsed} according to the terminology introduced by Gigli and De Philippis in \cite{GigliDeP}. Note that it follows from \cite{BruéSemola} that $N$ is an integer which coincides with the essential dimension of $(X,\dist,\meas)$.

We take the weight $w_o = V(o,\dist(o,\cdot))^{2/(N-2)} \dist(o,\cdot)^{-2N/(N-2)}$ which corresponds to the case $p=2$ in Theorem \ref{th:weightedSobolev}. Note that \eqref{eq:Ahlfors} together with Bishop-Gromov's theorem implies that $w_o$ is bounded from above by $C_o^{2/(N-2)}$, thus $L^2(\meas) \subset L^2(\mu)$.

Set $H^{1,2}_{loc}(X,\dist,\meas) = \{f \in L^2_{loc}(\meas) \, : \, \phi f \in H^{1,2}(X,\dist,\meas) \, \, \, \, \forall \phi \in \Lip_{bs}(X) \}$ and note that as an immediate consequence of \eqref{eq:Ahlfors} combined  with Bishop-Gromov's theorem, $w_o$ is bounded from above and below by positive constants on any compact subsets of $X$, thus $f \in L^2_{loc}(\meas)$ if and only if $f \in L^2_{loc}(\mu)$.


Define a Dirichlet form $Q$ on $L^2(\mu)$ as follows. Set
$$\mathcal{D}(Q):=\{f \in L^2(\mu) \cap H^{1,2}_{loc}(X,\dist,\meas) : |\nabla f|_* \in L^2(\meas) \}$$
and
$$
Q(f) = \begin{cases}
\int_X |\nabla f|^2_* \di \meas & \text{if $f \in \cD(Q)$},\\
+ \infty & \text{otherwise.}
\end{cases}
$$
$Q$ is easily seen to be convex. Let us show that it is a $L^2(\mu)$-lower semicontinuous functional on $L^2(\mu)$. Let $\{f_n\}_n \subset \mathcal{D}(Q)$ and $f \in L^2(\mu)$ be such that $\|f_n - f\|_{L^2(\mu)}\to 0$. Let $K \subset X$ be a compact set. For any $i \in \N\backslash \{0\}$, set
$$
\phi_i(\cdot) = \max(0,1-(1/i)\dist(\cdot,K))
$$
and note that $\phi_i \in \Lip_{bs}(X)$, $0 \le \phi_i \le 1$, $\phi_i \equiv 1$ on $K$ and $|\nabla \phi_i|_* \le (1/i)$. Then for any $i$, the sequence $\{\phi_i f_n\}_n$ converges to $\phi_i f$ in $L^2(\meas)$. The $L^2(\meas)$-lower semicontinuity of the Cheeger energy and the chain rule for the slope imply
\begin{align*}
\int_K|\nabla f|_*^2 \di \meas & \le  \int_X|\nabla (\phi_i f)|_*^2 \di \meas \le \liminf_{n} \int_X|\nabla (\phi_i f_n)|_*^2 \di \meas \\
& \le \liminf_{n} \int_X|\nabla f_n|_*^2 \di \meas + \frac{2}{i} \liminf_{n} \int_X f_n |\nabla f_n|_* \di \meas + \frac{1}{i^2} \liminf_{n} \int_X f_n^2 \di \meas.
\end{align*}
Letting $i$ tend to $+\infty$, then letting $K$ tend to $X$, yields the result.

Then we can apply the general theory of gradient flows to define the semigroup $(h_{t}^{\mu})_{t>0}$ associated to $Q$ which is characterized by the property that for any $f \in L^2(X,\mu)$, $t \rightarrow h_{t}^{\mu}f$ is locally absolutely continuous on $(0,+\infty)$ with values in $L^2(X,\mu)$, and
\[
\frac{\di}{\di t}h_{t}^{\mu} f= - A h_{t}^{\mu} f \quad \text{for $\mathcal{L}^1$-a.e. $t\in (0,+\infty)$},
\]
where the self-adjoint operator $-A$ associated to $Q$ is defined on a dense subset $\cD(A)$ of $\cD(Q)=\{Q<+\infty\}$ and characterized by:
\[
Q(f,g)=\int_{X}(Af)g \di \mu \qquad \forall f \in \cD(A), \, \forall g \in \cD(Q).
\] 
Be aware that although $Q$ is defined by integration with respect to $\meas$, it is a Dirichlet form on $L^2(\mu)$, whence the involvement of $\mu$ in the above characterization.

Note that by the Markov property, each $h_{t}^{\mu}$ can be uniquely extended from $L^{2}(X,\mu) \cap L^{1}(X,\mu)$ to a contraction from $L^{1}(X,\mu)$ to itself.

We start with a preliminary lemma stating that a weighted Nash inequality also holds on the appropriate functional space when $\Ch$ is replaced by $Q$.

\begin{lemma}\label{lem:NashQ}
Let $(X,\dist,\meas)$ be a $\RCD(0,N)$ space with $N> 3$ satisfying \eqref{eq:growthcondition} and \eqref{eq:Ahlfors} for some $\eta>2$, $C_o>1$ and $r_o>0$. Then there exists a constant $C=C(N,\eta,\Theta_{inf},\Theta_{sup})>0$ such that:
$$
\|u\|_{L^2(\mu)}^{2 + \frac{4}{N}} \le C \|u\|_{L^1(\mu)}^{\frac{4}{N}} Q(u) \qquad \forall u \in L^1(\mu) \cap \mathcal{D}(Q).
$$
\end{lemma}

\begin{proof}
Let $u \in  L^1(\mu) \cap \mathcal{D}(Q)$. Then $u \in L^2_{loc}(\meas)$, $\phi u \in H^{1,2}(X,\dist,\meas)$ for any $\phi \in \Lip_{bs}(X)$ and $|\nabla u|_* \in L^2(\mu)$. In particular, if we take $(\chi_n)_n$ as in the proof of Lemma \ref{lem:standardlemma}, for any $n \in \N$ we get that $\chi_n u \in H^{1,2}(X,\dist,\meas)$ and consequently there exists a sequence $(u_{n,k})_k \subset \Lip_{bs}(X)$ such that $u_{n,k} \to \chi_n u$ in $L^2(\meas)$ and $\int_X |\nabla u_{n,k}|^2 \di \meas \to \int_X |\nabla (\chi_n u)|_*^2 \di \meas$. Apply Theorem \ref{th:weightedNash} to the functions $u_{n,k}$ to get
\begin{equation}\label{eq:41}
\|u_{n,k}\|_{L^2(\mu)}^{2+\frac{4}{N}} \le C \|u_{n,k}\|_{L^1(\mu)}^{\frac{4}{N}} \int_X |\nabla u_{n,k}|^2 \di \meas
\end{equation}
for any $k \in \N$. As the $u_{n,k}$ and $\chi_n u$ have bounded support, and thanks to \eqref{eq:Ahlfors} which ensures boundedness of $w_o$, the $L^2(\meas)$ convergence $u_{n,k} \to \chi_n u$ is equivalent to the $L^2_{loc}(\meas)$, $L^2_{loc}(\mu)$, $L^2(\mu)$ and $L^1(\mu)$ convergences. Therefore, passing to the limit $k \to +\infty$ in \eqref{eq:41}, we get
$$
\|\chi_n u\|_{L^2(\mu)}^{2+\frac{4}{N}} \le C \|\chi_n u\|_{L^1(\mu)}^{\frac{4}{N}} \int_X |\nabla (\chi_n u)|_*^2 \di \meas.
$$
By an argument similar to the proof of Lemma \ref{lem:standardlemma}, we can show that
$$ 
\limsup\limits_{n \to +\infty} \int_X |\nabla (\chi_n u)|_*^2 \di \meas \le \int_X |\nabla u|_*^2 \di \meas.
$$
And monotone convergence ensures that $\|\chi_n u\|_{L^2(\mu)} \to \|u\|_{L^2(\mu)}$ and $\|\chi_n u\|_{L^1(\mu)} \to \|u\|_{L^1(\mu)}$, whence the result.
\end{proof}

Let us apply Lemma \ref{lem:NashQ} to get a bound on the heat kernel of $Q$.

\begin{theorem}[Bound of the weighted heat kernel]\label{th:boundweightedheatkernel}
Let $(X,\dist,\meas)$ be a $\RCD(0,N)$ space with $N>3$ satisfying the growth condition \eqref{eq:growthcondition} for some $\eta>2$ and the uniform local $N$-Ahlfors regular property \eqref{eq:Ahlfors} for some $C_o>1$ and $r_o>0$. Then there exists $C=C(N,\eta,\Theta_{inf},\Theta_{sup})>0$ such that
\begin{equation}\label{eq:111}
\|h_{t}^{\mu}\|_{L^1(\mu) \rightarrow L^{\infty}(\mu)} \le \frac{C}{t^{N/2}}, \qquad \forall t>0.
\end{equation}
Moreover, for any $t>0$, $h_{t}^{\mu}$ admits a kernel $p_{t}^{\mu}$ with respect to $\mu$ such that for some $C=C(N,\eta,\Theta_{inf},\Theta_{sup})>0$,
\begin{equation}\label{eq:222}
p_{t}^{\mu}(x,y) \le \frac{C}{t^{N/2}} \qquad \forall x, y \in X.
\end{equation}
\end{theorem}

To prove this theorem we follow closely the lines of \cite[Th.~4.1.1]{Saloff-Coste}. The constant $C$ may differ from line to line, note however that it will always depend only on $\eta$, $N$, $\Theta_{inf}$ and $\Theta_{sup}$.

\begin{proof}
Let $u \in L^1(\mu)$ be such that $\|u\|_{L^1(\mu)} = 1$. Let us show that $\|h_{t}^{\mu}u\|_{L^2(\mu)} \le C t^{-N/4}$ for any $t>0$. First of all, by density of $\Lip_{bs}(X)$ in $L^1(\mu)$, we can assume $u \in \Lip_{bs}(X)$ with $\|u\|_{L^1(\mu)} = 1$. Furthermore, since for any $t>0$, the Markov property ensures that the operator $h_t^\mu : L^1(\mu) \cap L^2(\mu) \to \mathcal{D}(Q)$ extends uniquely to a contraction oerator from $L^1(\mu)$ to itself, we have $h_t^\mu u \in L^1(\mu) \, (\cap \, \mathcal{D}(Q))$ and $\|h_t^{\mu}u\|_{L^1(\mu)} \le 1$. Therefore, we can apply Lemma \ref{lem:NashQ} to get:
$$
\|h_t^\mu u\|_{L^2(\mu)}^{2+\frac{4}{N}} \le C Q(h_t^\mu u) \qquad \forall t >0.
$$
As
$ \displaystyle
\int_X |\nabla h_t^\mu u|_*^2 \di \meas = \int_{X} (A h_{t}^{\mu}u) h_{t}^{\mu}u \di \mu = - \int_{X} \left( \frac{\di}{\di t} h_{t}^{\mu}u \right) h_{t}^{\mu}u \di \mu = - \frac{1}{2} \frac{\di}{\di t} \| h_{t}^{\mu}u \|_{L^2(\mu)}^2,
$
we finally end up with the following differential inequality:
$$
\|h_{t}^{\mu}u\|_{L^2(\mu)}^{2 + 4/N} \le - \frac{C}{2} \frac{\di}{\di t} \| h_{t}^{\mu}u \|_{L^2(\mu)}^2 \qquad \forall t>0.
$$
Writing $\phi(t) = \|h_{t}^{\mu}u\|_{L^2(\mu)}^2$ and $\psi(t)=\frac{N}{2} \phi(t)^{-2/N}$ for any $t>0$, we get $\frac{2}{C} \le \psi'(t)$ and thus $\frac{2}{C} t \le \psi(t) - \psi(0)$. As $\psi(0)=\frac{N}{2}\|u\|_{L^2(\mu)}^{-4/N} \ge 0$, we obtain $\frac{2}{C} t \le \psi(t)$, leading to
$$ \|h_{t}^{\mu}u\|_{L^2(\mu)} \le \frac{C}{t^{N/4}}.$$
We have consequently $\|h_{t}^{\mu}\|_{L^1(\mu) \rightarrow L^2(\mu)} \le \frac{C}{t^{N/4}}$. Using the self-adjointness of $h_{t}^\mu$, we deduce $\|h_{t}^{\mu}\|_{L^2(\mu) \rightarrow L^\infty(\mu)} \le \frac{C}{t^{N/4}}$ by duality. Finally the semigroup property 
\[
\|h_{t}^{\mu}\|_{L^1(\mu) \rightarrow L^\infty(\mu)} \le \|h_{t/2}^{\mu}\|_{L^1(\mu) \rightarrow L^2(\mu)} \|h_{t/2}^{\mu}\|_{L^2(\mu) \rightarrow L^\infty(\mu)}\]
implies \eqref{eq:111}. Then the existence of a measurable kernel $p^\mu_t$ of $h^\mu_t$ for any $t>0$ together with the bound \eqref{eq:222} is a direct consequence of Lemma \ref{lem:NashQ}, thanks to \cite[Th. (3.25)]{CKS}.

\end{proof}

\section{A non-smooth example}

To conclude, let us provide an example beyond the scope of smooth Riemannian manifolds to which Theorem \ref{th:weightedSobolev} applies. For any positive integer $n$, let $0_n$ be the origin of $\R^n$.





In \cite{Hattori}, Hattori built a complete four dimensional Ricci-flat manifold $(M,g)$ satisfying \eqref{eq:growthcondition} for some $\eta \in (3,4)$ and whose set of isometry classes of tangent cones at infinity $\mathcal{T}(M,g)$ is homeomorphic to $\mathbb{S}^1$. Of particular interest to us is one specific element of $\mathcal{T}(M,g)$, namely $(\R^3,\dist_0^\infty,0_3)$, where $\dist_0^\infty$ is the completion of the Riemannian metric $f g_{e}$ defined on $\R^3\backslash\{0_3\}$ as follows: $g_e$ is the Euclidean metric on $\R^3$ and for any $x=(x_1,x_2,x_3) \in \R^3 \backslash\{0_3\}$,
\[
f(x) = \int_0^\infty b_x(t) \di t \quad \text{with} \quad b_x(t)=\frac{1}{\sqrt{(x_1-t^\alpha)^2 + x_2^2 + x_3^2}} \, ,
\]
for some $\alpha>1$. Since $b_x(t) \sim t^{-\alpha}$ when $t \to +\infty$ and $b_x(t) \sim |x|^{-1}$ when $t \to 0$ for any $x \neq 0_3$, then $f(x)$ has no singularity on $\R^3 \backslash \{0\}$; however, $b_{0_3}(t)= t^{-\alpha}$ is not integrable on any neighborhood of $0$, so $f(x)$ has a singularity at $x=0_3$. In particular, $(\R^3,\dist_0^\infty,0_3)$ is a singular space with a unique singularity at $0_3$. Hattori proved that this space is neither a metric cone nor a polar metric space.

Let $\dist_g, v_g$ be the Riemannian distance and Riemannian volume measure associated to $g$, and $o \in M$ such that $(\R^3,\dist_0^\infty,0_3)$ is a tangent cone at infinity of $(M,\dist_g,o)$. Following a classical method (see e.g.~\cite{CheegerColding1}), one can equip $(\R^3,\dist_0^\infty,0_3)$ with a limit measure $\mu$ such that for some infinitesimal sequence $(\eps_i)_i \in (0,+\infty)$ the rescalings $(M,\dist_{g_i},\underline{v}_{g_i},o)$, where $g_i=\eps_i^2g$ and $\underline{v}_{g_i} = v_{g_i}(B_{1/\eps_i}(o))^{-1}v_{g_i}$, converge in the pointed measured Gromov-Hausdorff sense to $(\R^3,\dist_0^\infty,\mu,0_3)$. As $(M,g)$ is Ricci-flat, so are any of its rescalings, in particular they are all $\RCD(0,4)$ spaces. The stability of the $\RCD(0,4)$ condition with respect to pointed measured Gromov-Hausdorff convergence implies that $(\R^3,\dist_0^\infty,\mu,0_3)$ is $\RCD(0,4)$ too. Let us prove that $(\R^3,\dist_0^\infty,\mu,0_3)$ also satisfies \eqref{eq:growthcondition}. Set
$$
\Theta_{inf}(M,g) := \liminf\limits_{r \to +\infty} \frac{v_g(B_r(o))}{r^\eta} \quad \text{and} \quad \Theta_{sup}(M,g):= \limsup\limits_{r \to +\infty} \frac{v_g(B_r(o))}{r^\eta} \, \cdot
$$
Then for any $r>0$,
\begin{align*}
\frac{\mu(B_r(0_3))}{r^\eta} & = \lim\limits_{i \to +\infty} \frac{\underline{v}_{g_i}(B_r^{i}(o))}{r^\eta} = \lim\limits_{i \to +\infty} \frac{v_{g_i}(B_r^{i}(o))}{v_{g_i}(B_1^{i}(o)) r^\eta} = \lim\limits_{i \to +\infty} \frac{v_{g}(B_{r/\eps_i}(o))}{v_{g}(B_{1/\eps_i}(o)) r^\eta} \\
& = \lim\limits_{i \to +\infty} \frac{v_{g}(B_{r/\eps_i}(o))}{(r/\eps_i)^\eta} \frac{(1/\eps_i)^\eta}{v_{g}(B_{1/\eps_i}(o))} \, ,
\end{align*}
so
\[
\Lambda:=\frac{\Theta_{inf}(M,g)}{\Theta_{sup}(M,g)} \le \frac{\mu(B_r(0_3))}{r_j^\eta} \le \Lambda^{-1}
\]
from which \eqref{eq:growthcondition} follows with $\Theta_{inf}\ge \Lambda$ and $\Theta_{sup}\le \Lambda^{-1}$.


\begin{thebibliography}{}
\bibitem[ACDM15]{AmbrosioColomboDiMarino}
      \textsc{L. Ambrosio, M. Colombo, S. Di Marino}:
      \textit{Sobolev spaces in metric measure spaces: reflexivity and lower semicontinuity of slope},
      Adv. Studies in Pure Math. {\bf 67} (2015), 1--58.

\bibitem[AGS14b]{AmbrosioGigliSavare14}
	\textsc{L. Ambrosio, N. Gigli, G. Savar\'e}:
	\textit{Metric measure spaces with Riemannian Ricci curvature bounded from below},
	Duke Math. J. \textbf{163} (2014), 1405--1490.

\bibitem[BBGL12]{BakryBolleyGentilMaheux}
       \textsc{D. Bakry, F. Bolley, I. Gentil, P. Maheux}:
       \textit{Weighted Nash inequalities},
        Revista Matematica Iberoamericana \textbf{28}(3) (2012), 879--906.

\bibitem[BCLS95]{BakryCoulhonLedouxSaloff-Coste}
       \textsc{D. Bakry, T. Coulhon, M. Ledoux, L. Saloff-coste}:
       \textit{Sobolev inequalities in disguise}, 
        Indiana Univ. Math. J \textbf{44}(4) (1995), 1033--1074.

\bibitem[BS18]{BruéSemola}
	\textsc{E. Brué, D. Semola}:
	\textit{Constancy of the dimension for $\RCD(K,N)$ spaces via regularity of Lagrangian flows},
     to appear in Comm. Pure and Applied Math., ArXiV preprint 1804.07128v1 (2018).

\bibitem[BBI01]{BuragoBuragoIvanov} 
	\textsc{D. Burago, Y. Burago, S. Ivanov}:
	\textit{A course in metric geometry},
	Graduate Studies in Mathematics, 33, American Mathematical Society, Providence, RI, xiv+45 pp (2001).

\bibitem[CKS87]{CKS}
          \textsc{E. Carlen, S. Kusuoka, D.W. Stroock}:
          \textit{Upper Bounds for symmetric Markov transition functions},
          Annales de l'I.H.P. Probabilités et statistiques \textbf{23} (1987), 245--287.   

\bibitem[CM17]{CavallettiMondino2}
          \textsc{F. Cavalletti, A. Mondino}:
          \textit{Sharp geometric and functional inequalities in metric measure spaces with lower Ricci curvature bounds},
          Geom. Topol. \textbf{21}(1) (2017), 603--645.   

\bibitem[Ch99]{Cheeger}
         \textsc{J. Cheeger}: \textit{Differentiability of Lipschitz functions on metric measure spaces}. 
         Geom. Funct. Anal. \textbf{9} (1999), 428--517.
         
\bibitem[CC97]{CheegerColding1}
          \textsc{J. Cheeger, T. H. Colding}:
          \textit{On the structure of spaces with Ricci curvature bounded below, I,}  J. Differential Geom. \textbf{46} (1997), 406--480.

\bibitem[DG18]{GigliDeP}
        \textsc{G. De Philippis, N. Gigli}:
        \textit{Non-collapsed spaces with Ricci curvature bounded from below},
 J. Éc. polytech. Math. \textbf{5} (2018), 613–650.



\bibitem[G18]{Gigli}
        \textsc{N. Gigli}:
        \textit{Nonsmooth differential geometry --
An approach tailored for spaces with Ricci curvature bounded from below},
        Mem. Am. Math. Soc. \textbf{251} (2018), 1--161

\bibitem[GHL09]{Grigor'YanHuLau}
	\textsc{A. Grigor'yan, J. Hu, K.-S. Lau}:
	\textit{Heat kernels on metric spaces with doubling measure}.
	Fractal Geometry and Stochastics {I}{V}, Progress in Probability \textbf{61} (2009), 3--44. 

\bibitem[GS05]{Grigor'yanSaloff-Coste}
	\textsc{A. Grigor'yan, L. Saloff-Coste}:
	\textit{Stability results for Harnack inequalities},
	 Ann. Inst. Fourier \textbf{55}(3) (2005), 825--890. 

\bibitem[HK00]{HajlaszKoskela}
       \textsc{P. Hajlasz, P. Koskela}:
        \textit{Sobolev met Poincaré}, 
        Me. Am. Math. Soc. \textbf{145} (2000), 1--101. 
        
\bibitem[Ha17]{Hattori}
       \textsc{K. Hattori}:
        \textit{The nonuniqueness of the tangent cones at infinity of Ricci-flat manifolds}, 
        Geom. Topol. \textbf{21}(5) (2017), 2683--2723. 

\bibitem[He11]{Hein}
       \textsc{H.-J. Hein}:
        \textit{Weighted Sobolev inequalities under lower Ricci curvature bounds},
        Proc. of the Am. Math. Soc. \textbf{139} (2011), 2943--2955.
        
\bibitem[HeK98]{HeinonenKoskela}
       \textsc{J. Heinonen, P. Koskela}:
        \textit{Quasiconformal maps in metric spaces with controlled geometry}, 
        Acta Math. \textbf{181} (1998), 1--101.





\bibitem[LY86]{LiYau}
       \textsc{P. Li, S.-T. Yau}:
       \textit{On the parabolic kernel of the Schrödinger operator}, 
       Acta Math. \textbf{156} (1986), 153-201.

\bibitem[LV09]{LottVillani}
\textsc{J. Lott, C. Villani}:
\textit{Ricci curvature for metric-measure spaces via optimal transport}, 
Ann. of Math. \textbf{169} (2009), 903--991.

\bibitem[LV07]{LottVillani2}
\textsc{J. Lott, C. Villani}:
\textit{Weak curvature conditions and functional inequalities}, 
J. Funct. Anal. \textbf{245} (2007), 311--333.

\bibitem[MS95]{MaheuxSaloffCoste}
       \textsc{P. Maheux, L. Saloff-Coste}:
       \textit{Analyse sur les boules d'un opérateur sous-elliptique},
        Mathematische Annalen \textbf{303}(4) (1995), 713-740. 

\bibitem[M09]{Minerbe}
       \textsc{V. Minerbe}:
        \textit{Weighted Sobolev inequalities and Ricci flat manifolds},
        G.A.F.A. \textbf{18}(5) (2009), 1696--1749.


\bibitem[Oh17]{Ohta17}
        \textsc{S.-I. Ohta}:
\textit{Some functional inequalities on non-reversible Finsler manifolds,}
      Proc. Indian Acad. Sci. Math. Sci. \textbf{127} (2017), 833--855.

\bibitem[Pa14]{Papadopoulos}
{\sc A.~Papadopoulos}: {\it Metric spaces, convexity and non-positive curvature}, Second edition. IRMA Lectures in Mathematics and Theoretical Physics \textbf{6}, European Mathematical Society (EMS), Zürich, 2014.


\bibitem[Pr15]{Profeta}
		\textsc{A.Profeta}:
		\textit{The sharp Sobolev inequality on metric measure spaces with lower Ricci curvature bounds},
		Pot. Anal. \textbf{43} (2015), 513--529.

\bibitem[Raj12]{Rajala}
        \textsc{T. Rajala}:
        \textit{Local Poincar\'e inequalities from stable curvature conditions on metric spaces},
        Calc. Var. Partial Differential Equations \textbf{44}(3) (2012), 477--494.

\bibitem[SC02]{Saloff-Coste}
{\sc L.~Saloff-Coste}: {\it Aspect of Sobolev-type inequalities}, London Mathematical Society Lecture Note Series (No. 289), Cambridge University Press, 2002.

\bibitem[Sc79]{Schep}
	\textsc{A.~R.~Schep}:
	\textit{Kernel operators},
	 Indagationes Mathematicae \textbf{82}(1) (1979), 39--53.

\bibitem[St95]{Sturm95}
	\textsc{K.-T. Sturm}:
	\textit{Analysis on local Dirichlet spaces. II. Upper Gaussian estimates for the fundamental
solutions of parabolic equations},
	 Osaka J. Math. \textbf{32}(2) (1995), 275--312.

\bibitem[St96]{Sturm96}
	\textsc{K.-T. Sturm}:
	\textit{Analysis on local Dirichlet spaces. III. The parabolic Harnack inequality},
	J. Math. Pures Appl. \textbf{75}(9) (1996) 273--297.

\bibitem[St06]{Sturm06}
	\textsc{K.-T. Sturm}:
	\textit{On the geometry of metric measure spaces, I and II},
	Acta Math. \textbf{196} (2006), 65--131 and 133--177.

\bibitem[T]{T}
        \textsc{D. Tewodrose}:
        \textit{Weighted Sobolev inequalities and volume growth on metric measure spaces}, in preparation.

\bibitem[T20]{Tewodrose2}
        \textsc{D. Tewodrose}:
        \textit{Adimensional weighted Sobolev inequalities in PI spaces}, ArXiV preprint: 2006.10493, (2020).

\bibitem[Vi09]{Villani}
{\sc C.~Villani}: {\it Optimal transport. Old and new}, vol.~338 of Grundlehren
  der Mathematischen Wissenschaften, Springer-Verlag, Berlin, 2009.
\end{thebibliography}
\end{document}